\documentclass{amsart}
\usepackage[utf8]{inputenc}
\usepackage{amsmath,amsthm,amssymb,mathrsfs}
\usepackage[all,cmtip]{xy}
\usepackage{color}
\usepackage{tikz}
\usetikzlibrary{matrix,arrows}

\newtheorem{theorem}{Theorem}[section]
\newtheorem{lemma}[theorem]{Lemma}
\newtheorem{example}[theorem]{Example}
\newtheorem{proposition}[theorem]{Proposition}

\theoremstyle{definition}
\newtheorem{definition}[theorem]{Definition}

\theoremstyle{remark}

\newtheorem{remark}[theorem]{Remark}

\numberwithin{equation}{section}

\newcommand{\umon}[1]{\pmb{u}^{\pmb{#1}}}
\newcommand{\wmon}[1]{\pmb{w}^{\pmb{#1}}}
\newcommand{\xmon}[1]{x^{\pmb{#1}}}
\newcommand{\ymon}[1]{y^{\pmb{#1}}}
\newcommand{\red}[1]{\textcolor{red}{#1}}
\newcommand{\sm}[1]{{\scriptstyle (#1)}}

\newcommand{\EE}{{\mathbb{S}^2}}
\newcommand{\N}{\mathbb{N}}
\newcommand{\R}{\mathbb{R}}
\newcommand{\RR}{\mathbb{R}^2}
\newcommand{\C}{\mathbb{C}}
\newcommand{\eps}{\varepsilon}
\newcommand{\Cl}{\mathop{\rm Cl}\nolimits}
\newcommand{\Inte}{\mathop{\rm Int}\nolimits}

\title[Topological classification of limit periodic sets]{Topological classification of limit periodic sets of polynomial planar vector fields}
\author[A.~Belotto]{Andr\'e Belotto da Silva}
\author[J.~G.~Esp\'{i}n]{Jose Gin\'es Esp\'{i}n Buend\'{i}a}
\address{AB: Université Paul Sabatier, Institut de Mathématiques de Toulouse, 118 route de Narbonne, F-31062 Toulouse Cedex 9, France}
\email[A.~Belotto]{{andre.belotto\_da\_silva@math.univ-toulouse.fr}}
\address{JE: Universidad de Murcia, Departamento de Matem\'aticas, Campus de Espinardo, 30100 Murcia, Spain}
\email[J.~G.~Esp\'{i}n]{josegines.espin@um.es}



\begin{document}

\maketitle

\section*{Abstract}
We characterize the limit periodic sets of families of algebraic planar vector fields up to homeomorphisms. We show that any limit periodic set is topologically equivalent to a compact and connected semialgebraic set of the sphere {of dimension $0$ or $1$}. Conversely, we show that any compact and connected semialgebraic set of the sphere {of dimension $0$ or $1$} can be realized as a limit periodic set.

\setcounter{tocdepth}{3}

\section{Introduction}

The subject of this manuscript is the structure of limit periodic sets of planar polynomial vector fields, a central object in bifurcation theory and in the treatment of the Hilbert $16^{th}$ problem (see Roussarie's book \cite{r98} or Il'yashenko and Yakovenko's book \cite{iy95}). For example, the program of Dumortier, Roussarie and Rousseau \cite{drr94} to solve the existential part of the $16^{th}$ Hilbert problem for quadratic vector fields is divided in $121$ case-by-case analysis based on the limit periodic sets. Following the spirit of \cite{jl07}, {\cite{jp09}} or \cite{lr04}, our objective is to characterize topologically all limit periodic sets of polynomial families of planar vector fields.

We consider a real algebraic manifold $\Lambda$ of dimension $n \geq 1$, which we call \textit{parameter space}. A \textit{family of planar vector fields} $(X_{\lambda})_{\lambda \in \Lambda}$, is a vector field $X_{\lambda}$ defined on $\mathbb{R}^2 \times \Lambda$ which is tangent to the fibres of the projection $\pi: \mathbb{R}^2 \times \Lambda \to \Lambda$. For any parameter $\lambda_0 \in \Lambda$, we denote by $X_{\lambda_0}$ the restriction of $X_{\lambda}$ to $\mathbb{R}^2 \times \{\lambda_0\}$, which we identify with $\mathbb{R}^2$. We say that the family $(X_{\lambda})_{\lambda \in \Lambda}$ is \textit{polynomial} if for each $\lambda_0 \in \Lambda$ there exist local coordinate systems $x= (x_1,x_2)$ of $\mathbb{R}^2$ and $\lambda= (\lambda_1, \ldots, \lambda_n)$ centered at $\lambda_0$ such that $X_{\lambda}(x) = A_1(x,\lambda) \partial_{x_1} + A_2(x,\lambda) \partial_{x_2}$ where $A_1$ and $A_2$ are polynomials. 

Given any polynomial vector field $X$ on $\R^2$, we shall extend it to an analytic vector field, which we denote by $\hat{X}$, in the sphere $\mathbb{S}^2$ via a Bendixson compactification (see details in Section \ref{bendixson}). Also, for every $A \subset \RR$, we will write $\hat{A}$ to denote the closure of $A$ seen as a subset in the one-point compactification $\mathbb{S}^2$ of $\mathbb{R}^2$. 

We recall the definition of limit periodic sets, which was first introduced by Fran\c{c}oise and Pugh \cite[pp. 141]{fp86}. \begin{definition}
A \textit{limit periodic set} for a polynomial family of planar vector fields $(X_{\lambda})_{\lambda \in \Lambda}$ at the parameter $\lambda_0$ is a closed set $\Gamma \subset \mathbb{R}^2$ for which there exist a sequence $(\lambda_{n})_{n}$ in the parameter space $\Lambda$ and a sequence $(\gamma_n)_n$ of topological circles in $\R^2$ such that $(\lambda_n)_n$ converges to $\lambda_0$ in $\Lambda$, $(\gamma_{n})_n$ converges to $\hat{\Gamma}$ in the Hausdorff topology of $\mathbb{S}^2$ and, for every $n$, the vector field $X_{\lambda_n}$ has $\gamma_{n}$ as a limit cycle.
\end{definition}

In terms of the structure of limit periodic sets, it is well-known that the Poincar\'{e}-Bendixson Theorem implies:

\begin{proposition}(See \cite[Proposition 1]{fp86})\label{prop:classical}
Let $(X_{\lambda})_{\lambda \in \Lambda}$ be a polynomial family of planar vector fields and $\Gamma$ be a limit periodic set at the parameter $\lambda_0$. Then $\hat{\Gamma}$ is one of the following: (i) a singular point of $\hat{X}_{\lambda_0}$; (ii) a periodic orbit of $\hat{X}_{\lambda_0}$; (iii) a polycycle of $\hat{X}_{\lambda_0}$ (that is, a cyclic ordered collection of singular points $a_1, \ldots, a_k$ and arcs, given by integral curves, connecting them in the specific order: the $j$th arc connects $a_j$ with $a_{j+1}$); (iv) a degenerate limit cycle, that is, it contains non-isolated singularities of the vector field $\hat{X}_{\lambda_0}$.
\end{proposition}

While the above Proposition provides some key information about the nature of limit periodic sets, it does not fully characterize them. The present paper intends to fulfils this gap. A first characterization {was} provided by Panazzolo and Roussarie in \cite{pr95}, under the additional hypothesis that the first jet of the singular points of $X_{\lambda_0}$ is non-vanishing. {In the same paper, the authors also showed a first example of a limit periodic set which is not topologically in the list of possibilities of the Poincar\'e-Bendixson Theorem \cite[Example 3.1]{pr95}.} {Going further, in \cite{b12}, the first author has presented a class of examples of limit periodic sets which, topologically, are not in the list of possibilities of Poincar\'e-Bendixson Theorem either}. Here, we improve and generalize the construction of \cite{b12} in order to prove the converse of our main result:

\begin{theorem}\label{mainth}
Let $(X_{\lambda})_{\lambda \in \Lambda}$ be a polynomial family of planar  vector fields and $\Gamma$ a limit periodic set. Then there exists a homeomorphism $\varphi: \mathbb{S}^2 \to \mathbb{S}^2$ such that $\varphi(\hat{\Gamma})$ is a compact and connected semialgebraic set {of dimension $0$ or $1$}.

Conversely, if $\Gamma$ is a non-empty closed semialgebraic subset of $\RR$ {of dimension $0$ or $1$} whose compactification $\hat{\Gamma} \subset \EE$ is connected, there exists a polynomial family of planar vector fields $(X_{\lambda})_{\lambda \in \Lambda}$ {having $\Gamma$ as a limit periodic set}.
\end{theorem}

{The following example illustrates the construction performed in Section~\ref{converseproof} to prove the converse of Theorem~\ref{mainth}.}

{
\begin{example}\label{examplemain}
Let $\Gamma \subset \mathbb{R}^2$ be the semi-algebraic set given by:
\[
\Gamma = \left\{(x,y)\in \mathbb{R}^2;\, f(x,y)=y (x^2 + y^2 - 1)=0\text{ and } g(x,y)=x^2+y^2\leq 4  \right\}. 
\]
Following the notation of Subsection~\ref{genericcase}, we consider the set of points $$S = \{ (-2,0),(2,0), (0,1),(0,-1) \}$$ (where notice that $S = Gen(\Gamma) \cup Tr(\Gamma)$ and $NG(\Gamma)=\emptyset$, see Definition~\ref{def:specialPoints}). Now, consider the three variable polynomial
\[
h(x,y,\lambda) = f(x,y)^2 - \lambda \prod_{p\in S}\left(\|(x,y)-p\|^2-\lambda^2\right)
\]
where $\lambda$ will play the role of the parameter of the family of vector fields. Let $t \in \mathbb{R}^{+}$ and note that the level curves $Z_t = \{(x,y) \in \mathbb{R}^2;\, h(x,y,t)=0 \}$ are connected and converge (in the Hausdorff topology) to $\Gamma$ when $t$ goes to zero (c.f. Proposition~\ref{prop:CompactNonGeneric}; see Figure~\ref{exampleEasy}). It follows that the perturbation of the Hamiltonian vector field given by
\[
X_{\lambda} =\left( \frac{\partial h}{\partial y} + h \frac{\partial h}{\partial x}\right)  \partial_{x} + \left(-\frac{\partial h}{\partial x}  + h \frac{\partial h}{\partial y} \right) \partial_{y}
\]
is an polynomial family of planar vector fields which has $\Gamma$ as a limit periodic set for the parameter $\lambda_0=0$ (for every $t>0$, the set $Z_t$ is a limit cycle for $X_{t}$). 
\end{example}
}

\begin{figure}
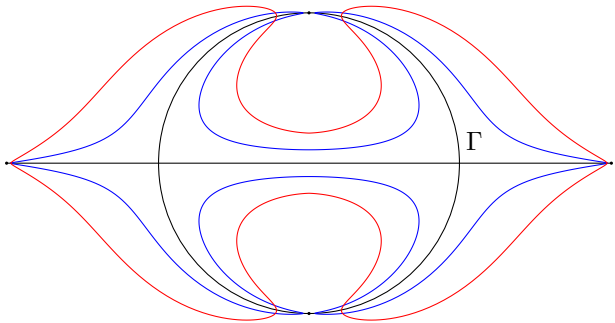


\caption{Limit cycles for $\lambda=0.001$ (red) and $\lambda=0.0001$ (blue) approaching the limit periodic set $\Gamma$.}
\label{exampleEasy}
\end{figure}

{The rest of the paper is divided as follows. In Subsection~\ref{ssec:Rks} we present some remarks about Theorem~\ref{mainth}; the aim of Section~\ref{directproof} is to prove the direct implication of Theorem~\ref{mainth} while Section~\ref{converseproof} deals with the converse one.}

\subsection{Remarks}\label{ssec:Rks}

\begin{itemize}

\item[(I)] {If we restrict our study to compact limit periodic sets of the plane, Theorem \ref{mainth} can be extended to the analytic category. More precisely, with the same ideas and techniques, it is not difficult to show that a compact limit periodic set for an analytic family of vector fields is topologically equivalent to a compact and connected semianalytic set {of dimension $0$ or $1$}; conversely, every compact and connected semianalytic set {of dimension $0$ or $1$} can be realized as a limit periodic set for an analytic family of vector fields.}

\item[(II)] {On the other hand, Theorem \ref{mainth} does not extend, in a trivial, to unbounded limit periodic sets for families of analytic vector fields. The difficulty relies on proving the converse part of the theorem. Let us first exemplify how our methods could be adapted to some unbounded analytic varieties: we claim that the set 
\[
\Gamma_1 = \{(x,y)\in \mathbb{R}^2;\, f_1(x,y)= y^2-\sin(x)^2=0\}
\]
can be realized as a limit periodic set for an analytic family. Indeed, it suffices to replace the function $h$ in Section~\ref{genericcase} by
\[
h(x,y,\lambda,\alpha) = f_1(x,y)^2 - \lambda\left(1- \alpha^2(x^2+y^2)   \right).
\]
We leave it to the reader to verify that the ideas of Sections~\ref{genericcase}~and~\ref{nongenericcase} can be adapted to this function. Nevertheless, it is unclear which connected subsets of
\[
\Gamma_2 = \{(x,y)\in \mathbb{R}^2;\, f_2(x,y) = y^3-y\sin(x)^2=0\}
\] 
can be realized as limit periodic sets. Technically, the difficulty is that our construction for $\Gamma_2$ would demand the use of transition points (defined in the last paragraph of Section \ref{properties_semi}); but the set of those transition points $Tr(\Gamma_2)$ would need to be infinite in this case.}

\item[(III)] A description of the limit periodic sets $\Gamma$ in the spirit of Proposition \ref{prop:classical} follows, {under the hypotehsis that$X_{\lambda_0}\not\equiv0$,} from the proof of Lemma \ref{final_lemma} below. More precisely, with the notation of the direct implication in Theorem \ref{mainth}, a limit periodic set $\Gamma$ must be a finite union $\bigcup_{i=1}^{m} S_i  \cup \bigcup_{j=1}^{n} \gamma_j$, for some $m,\,n\in\mathbb{N}$, where each $S_i$ is a connected semi-algebraic subsets of the set of singularities of $\hat{X}_{\lambda_0}$ and each $\gamma_j$ is a regular orbit of $\hat{X}_{\lambda_0}$ which converge to a singular points in $\bigcup_{i=1}^m S_i$. Even more, each $\gamma_j$ is characteristic in both extremes; that is, when the orbit is run in either negative or positive time, the orbit converges in a well-defined direction to a singular point of $\hat{X}_{\lambda_0}$ and, in a sufficiently small neighbourhood of that limit point, $\gamma_j$ is the frontier of a parabolic or hyperbolic sector.

\end{itemize}

\medskip
\noindent {\bf Acknowledgements.} We would like to thank the anonymous referee for the useful comments and the University of Toronto for its hospitality. The work of the first author is supported by LabEx CIMI. The second author is supported by Fundaci\'on S\'eneca through the program ``Contratos Predoctorales de Formaci\'{o}n del Personal Investigador'', grant 18910/FPI/13 and by the MINECO grants MTM2014-52920-P.

\section{Topology of limit periodic Sets}\label{directproof}

Along the Section, the reader is assumed to be familiar with some background in elementary Planar Qualitative Theory of Differential Equations; regarding this, \cite{dla06} is a good reference.

In Subsection~\ref{bendixson}, we recall the constructions of the Bendixson compactification of a polynomial vector field on $\R^2$. Subsection~\ref{semialg} is devoted to present the notion of real semialgebraic sets and some of their elementary properties and, finally, in Subsection~\ref{directProof}, we prove the direct part of Theorem~\ref{mainth}

\subsection{Bendixson Compactification}\label{bendixson} 
For the sake of simplicity, we present an adaptation of \cite[Section~1.1.3.2]{r98} using real analysis notation. 

The one-point compactification of the euclidean plane $\RR_{\infty}=\RR \cup \{\infty\}$ can be seen as a real analytic compact surface. Indeed, it is enough to consider the two local charts $(\R^2,z)$ and $(\R^2_{\infty}\setminus \{0\},Z)$ where $z:\RR \to \RR$ is the map given by the formula $z(x,y)=(r(x,y),s(x,y))=(x,y)$ for every $(x,y) \in \RR$ and $Z:\RR_{\infty} \setminus \{0\} \to \RR$ is given by $Z(x,y)=(u(x,y),v(x,y))=(x/(x^2 +y^2),-y/(x^2 +y^2))$ if $(x,y) \neq \infty$ and $Z(\infty)=(u(\infty),v(\infty))=0$. The equations for the changes of coordinates $z \circ Z^{-1}: \RR \setminus \{0\} \to \RR \setminus \{0\}$ and $Z \circ z^{-1}: \RR \setminus \{0\} \to \RR \setminus \{0\}$ are given by the analytic formulas $r = u/(u^2 + v^2)$ and $s = -v/(u^2 + v^2)$ and  $u = r/(r^2 + s^2)$ and $v = -s/(r^2 + s^2)$ respectively; this justifies that $\{(\R^2,z),(\R^2_{\infty}\setminus \{0\},Z)\}$ is an analytic atlas for $\RR_{\infty}$. {We denote by $\phi: \RR_{\infty}\to \RR_{\infty}$ the homeomorphism associated to this transition map (where $\phi(0)=\infty$ and $\phi(\infty)=0$; ; we will call $\phi$ the \textit{transition homeomorphism associated to the Bendixson compactification}.

We will also refer to $\RR_{\infty}$ as $\EE$ and we shall call it the Bendixson compactification of $\RR$. This notation is justify by the fact that $\RR_{\infty}$ is analytically diffeomorphic to the standard euclidean unit sphere $\mathbb{S}^2=\{(x,y,z) \in \R^3 : x^2 + y^2 + z^2 =1\}$: as a explicit analytic diffeomorphism we may consider the map $\psi:\mathbb{S}^2 \to \RR_{\infty}$ given by the formulas $\psi(0,0,1)=\infty$ and $\psi(x,y,z)=(x/(1 - z),y/(1 - z))$ if $(x,y,z) \neq (0,0,1)$. If we denote by $d_2$ the standard euclidean distance on $\R^3$, it follows that $\RR_{\infty}$, with {its} natural topology as the one-point compactification of $\RR$, is a metrizable space; and as a compatible distance we can take the one given by $d(a,b)=d_2(\psi^{-1}(a), \psi^{-1}(b))$ for every $a,b \in \RR_{\infty}$.

Let $P$ and $Q$ be real polynomials in two variables and consider the algebraic planar vector field given by $X=P  \partial_x + Q \partial_y$. If $d=\max\{\deg(P),\deg(Q)\}$, we may consider a vector field in the amplified euclidean plane $\RR_{\infty}$, $\hat{X}$, given by the formulas $\hat{X}(r,s)=\frac{1}{1 + (r^2 + s^2)^{d}} \left(  P(r,s) \partial_r  + Q(r,s) \partial_s \right) $ if $(r,s) \in \RR_{\infty} \setminus \{\infty\}$ and $\hat{X}(\infty)=0$. It is direct to show that $\hat{X}$ is well-defined and analytic in the whole $\RR_{\infty}$.

\subsection{Semialgebraic sets}\label{semialg}
Let $x=(x_1, \ldots, x_n)$ be a coordinate system of $\mathbb{R}^n$. Given any polynomial $f$ on $\R^n$ we shall say that $(f(x)=0)=\{x \in \R^n : \, f(x) =0\}$ is a \textit{algebraic set}. A more general concept is the following one. 

\begin{definition}(See \cite[Section 1]{bm88}). A subset $Z \subset \mathbb{R}^n$ is \emph{semialgebraic} if there exist polynomials $f_{i}$ and $g_{ij}$ on $\R^n$, $i=1,\ldots, p$ and $j=1, \ldots, q$, such that
\begin{equation*}\label{defsemalg}
Z= \bigcup_{i=1}^p \bigcap_{j=1}^q \{ x \in \R^n : \, f_{i}(x)=0 \text{ and } g_{ij}(x)>0\}.
\end{equation*}

A set $Z \subset \mathbb{S}^2 \subset \mathbb{R}^3$ is said to be semialgebraic if $Z$ is a semialgebraic set of $\mathbb{R}^3$.
\end{definition}

A fist collection of examples of semialgebraic sets is given by the finite unions of arcs linear by parts. For every $a,b \in \R^2$, we will denote $[{a},{b}]=\{{a} + s {b} : 0 \leq s \leq 1\}$, the straight arc joining ${a}$ and ${b}$. Let ${a}_1, \ldots,{a}_n$ be points in $\R^2$ and call $l_j=[{a}_j,{a}_{j+1}]$ for any $1 \leq j \leq n-1$. If $l_j \cap l_{j'} = \emptyset$ when $\left|j-j'\right|\neq 1$ and $l_j \cap l_{j+1} = \{{a}_{j+1}\}$ for any $1 \leq j \leq n-1$, we say that $L=\cup_{j=1}^{n-1}{l_j}$ is an \textit{arc linear by parts}. The points ${a}_1$ and ${a}_n$ are said to be the \textit{endpoints} of $L$. 

We next present a family of subsets of $\EE$ which are topologically equivalent to semialgebraic sets. 

Given any positive integer $n \in \N$, we say that a topological space is an \textit{$n$-star} if it is homeomorphic to $S_n=\{z \in \mathbb{C}: z^{n} \in [0,1]\}$. If $Z$ is an $n$-star and $h:S_n \to X$ is a homeomorphism, then the image of the origin under $h$ is called a \textit{vertex of the star} while the components of $Z \setminus \{h(0)\}$ are called the \textit{branches of the star}. Note that the vertex and the branches of a star are uniquely defined except in the cases $n=1,2$, when $Z$ is just a closed arc and the vertexes are its endpoints (for $n=1$) or its interior points (for $n=2$). We shall also adopt the convention of calling any singleton a $0$-star (the point being its vertex). When $Y$ is a topological space and ${a}$ is a point in $Y$ which posses a neighbourhood $Z \subset Y$ being an \emph{$n$-star} with ${a}$ as vertex, we will say that ${a}$ is a \textit{star point in $Y$ (of order $n$)}; if all the points in $Y$ are star points, we will say that $Y$ is a \textit{generalized graph}.

An important family of examples of generalized graphs is given by the set of zeros of planar analytic maps. 

\begin{proposition}\label{zeros} Let $f:U \to \R$ be an analytic map in an open subset $U \subset \RR$ and $Z=\{z \in U: f(z)=0\}$ be the set of zeros of $f$. Then either $Z$ is a whole connected component of $U$ or $Z$ is a generalized graph. \end{proposition}
\begin{proof} The result follows as a consequence of the Weierstrass Preparation Theorem and the theory of Puiseux Series (for a detailed record of the proof, see, for example,~\cite[p. 687]{jl07}). 
\end{proof}

We focus on a special subfamily of generalized graphs which shall play an important part in the {proof} of Theorem~\ref{mainth}. 

\begin{lemma}\label{top_eq_semi} Let ${L} \subset \mathbb{S}^2$ be a connected generalized graph. Then ${L}$ is topologically equivalent to a semialgebraic set{; that is, there exists a homeomorphism of $\mathbb{S}^2$ onto itself taking $L$ to a semialgebraic set.} \end{lemma} 
\begin{proof} Let us start by noticing that, apart from composing a rotation with the {transition} homeomorphism $\phi: \mathbb{R}^2_{\infty} \to \mathbb{R}^2_{\infty}$ {associated to} the Bendixson compactification (see Section~\ref{bendixson}), we may suppose that there exists a compact and connected set $\Gamma \subset \mathbb{R}^2$ such that its completion in $\mathbb{R}^2_{\infty}$ is equal to ${L}$. Also, if we call $T \subset \Gamma$ the subset of points which are not star points of order $2$, then $T$ is {a finite set}. It is then enough to prove the result under the hypothesis of $\Gamma$ being non-empty.

If $T$ is empty, there is nothing to say: ${L}$ is homeomorphic to $\{(x_1,x_2) \in \R^2 : x_1^2 + x_2^2 =0\}$~\cite[Theorem~2, p. 180]{KU}. Otherwise, let us say that $T=\{{a}_1, \ldots, {a}_m\}$ for some $m \geq 1$. For every $1 \leq j \leq m$ we can take a neighbourhood $B_{j} \subset \RR$ of ${a}_j$ such that $B_{j} \cap T = \{{a}_j\}$ and $B_{j} \cap \Gamma$ is an $n_j$-star. Without lost of generality we can also assume that, for every $1 \leq j \leq m$, $B_{j}$ is a standard euclidean compact ball of center ${a}_j$ in $\RR$, that $\partial B_{j}$ meets $\Gamma$ in exactly $n_j$ points ${b}_{j,1}, \ldots, {b}_{j,n_j}$ and, as a consequence, $B_{j} \cap \Gamma$ is homeomorphic to $M_j=\cup_{k=1}^{n_j} [{a}_j,{b}_{j,k}]$. Now any of the components of $\Gamma \setminus \cup_{j=1}^{m} B_j$ is a generalized graph consisting only of star points of order $2$, let us say $U_1, \ldots, U_{\tau}$ are those components. For any $1 \leq k \leq \tau$, we can take an arc linear by parts $N_k$ whose endpoints coincide with the points in $\Cl(U_k) \setminus U_k$ and such that $\cup_{j=1}^{m}M_j \, \cup \,\cup_{k=1}^{\tau} N_k$ is homeomorphic to $\Gamma$. This last homeomorphism can be extended to a homeomorphism from the {sphere} to the {sphere} (see, for example,~\cite[Theorem~1]{Adkisson1940}). 
\end{proof}

\subsection{Topological properties of periodic limit sets}\label{directProof} 

Let $\left\{X_{\lambda}\right\}_{\lambda \in \Lambda}$ be a polynomial family of planar vector fields and, for every $\lambda \in \Lambda$, let $\hat{X}_{\lambda}$ be the analytic vector field on $\EE$ described by the Bendixson compactification as in Section~\ref{bendixson} {(we remark that $p_N=(0,0,1)$ is a singular point for every $\hat{X}_{\lambda}$)}. Together with the family $\{\hat{X}_{\lambda}\}_{\lambda \in \Lambda}$ we may consider the associated analytic flow $\Phi: \R \times \EE \times \Lambda \to \EE$.

The continuity of the flow already gives some topological and dynamical obstructions for the limit periodic sets: a limit periodic set at a parameter $\lambda_0$ must be invariant for $X_{\lambda_0}$ and its compactification by one point must be connected. 

\begin{lemma}\label{InvConn} If $\Gamma$ is a limit periodic set at the parameter $\lambda_0$, then $\hat{\Gamma}$ is connected and invariant for $\hat{X}_{\lambda_0}$ (equivalently, $\Gamma$ is invariant for $X_{\lambda_0}$).
\end{lemma}
\begin{proof} Let us start fixing a sequence in $\Lambda$ {converging to $\lambda_0$}, $(\lambda_n)_n$, and a sequence of topological circles in $\RR$, $(\gamma_n)_n$, such that $(\hat{\gamma}_n)_n$ converges to $\hat{\Gamma}$ in the Hausdorff topology of $\mathbb{S}^2$ and, for every $n$, $\gamma_n$ is a limit cycle of $X_{\lambda_n}$.

Firstly, we consider points ${a} \in \Gamma$ and ${b} = {\Phi}(s,{a},\lambda_0)$ for some $s \in \R$ and a sequence of points ${a}_n \in \gamma_n$ converging to ${a}$. By the continuity of $X_{\lambda}$, the points ${\Phi}(s,{a}_n,\lambda_n)$ converge to ${b}$ so ${b} \in {\hat{\Gamma}}$. 

Next, to obtain a contradiction, let us suppose that $\hat{\Gamma}$ is not connected and choose two disjoint open sets $V_1$ and $V_2$ of {$\EE$} which disconnect {$\hat{\Gamma}$}. Since $\gamma_n \to \hat{\Gamma}$ in the Hausdorff topology, we conclude that $\gamma_n \subset V_1 \cup V_2$, $\gamma_n \cap V_1 \neq \emptyset$ and $\gamma_n \cap V_2 \neq \emptyset$, for $n$ sufficiently big. But this implies that $\gamma_n$ is disconnected, which is impossible. \end{proof}

From the analyticity of the flow $\Phi$ (we only need to use that it is of class $C^1$), the following important local property is established: any limit periodic set can meet at most once with any traversal. We formalize this property below. 

Let $I \subset \R$ be an open interval and $\lambda \in \Lambda$. An embedding $\sigma: I \to \EE$ of class $C^1$ is called a \textit{transverse section} of $\hat{X}_{\lambda}$ if, for any $s\in I$, the vectors {$\dot{\sigma}(s)$} and {$\hat{X}_{\lambda}(\sigma(s))$} are linearly independent. We shall also refer to {$\sigma(I)$} as a \textit{transverse section} of $\hat{X}_{\lambda}$.  

If ${a} \in \EE$ is a regular point of $\hat{X}_{\lambda_0}$, then we can always find a positive real number $\eps>0$ and an analytic embedding $\sigma: (-\eps,\eps) \to \EE$ being a transverse section of $\hat{X}_{\lambda_0}$ with $\sigma(0)={a}$. On the other hand, given any transverse section of $\hat{X}_{\lambda_0}$, $\sigma: I \to \EE$, it is clear that for any $t \in I$ we can take $I(t)$, an open neighbourhood of $t$ in $I$, and $\Lambda(\lambda_0)$, a neighbourhood of $\lambda_0$ in $\Lambda$, such that the restriction of $\sigma$ to $I(t)$ is a transverse section of $\hat{X}_{\lambda}$ for every $\lambda \in \Lambda(\lambda_0)$. These observations, together with the \textit{Flow Box Theorem} and the fact that any periodic orbit of a $C^1$ vector fields on the sphere can meet any transverse section only once, give the following result.

\begin{lemma}\label{transversal}(see~\cite[Lemma~2, p. 20]{r98}) Let $\Gamma$ be a limit periodic set at the parameter $\lambda_0$. Then any transverse section of $\hat{X}_{\lambda_0}$ meets $\hat{\Gamma}$ at most once.
\end{lemma} 

The last ingredient we need is the well-known behaviour of analytic vector fields on the neighbourhood of isolated singular points, the so-called \textit{finite sectorial decomposition property} (see~\cite[pp. 17--19]{dla06}): every sufficiently small neighbourhood of an isolated singular point of a planar analytic vector field is either a center, a focus or a finite union of hyperbolic, parabolic and elliptic sectors.

We are now ready to prove the direct implication of Theorem~\ref{mainth}. The work is done by the combination of Lemma~\ref{top_eq_semi} and the following result.

\begin{lemma}\label{final_lemma} {Suppose that} $\Gamma$ is a limit periodic set at the parameter $\lambda_0$ {and that $X_{\lambda_0}\not\equiv 0$. T}hen $\hat{\Gamma}$ is a connected generalized graph.
\end{lemma}

\begin{proof}
According with Lemma~\ref{InvConn}, $\hat{\Gamma}$ is a connected subset of $\mathbb{S}^2$ which is a union of orbits of $\hat{X}_{\lambda_0}$. Therefore, we only need to prove that all the points of $\hat{\Gamma}$ are star points. We fix a point $a \in \hat{\Gamma}$ and we distinguish three cases.

If $a$ is a regular point of $\hat{X}_{\lambda_0}$, the \textit{Flow Box Theorem} and Lemma~\ref{transversal} imply the existence of a  neighbourhood of $a$, $U_a$, such that $\hat{\Gamma} \cap U_a $ is a $2$-star. 

Let us now assume that $a$ is an isolated singular point of $\hat{X}_{\lambda_0}$ and let $U_a$ be a neighbourhood of $a$ such that every point in $U_a \setminus \{a\}$ is a regular point of $\hat{X}_{\lambda_0}$. If $a$ is a center (respectively a node or a focus) for $\hat{X}_{\lambda_0}$, we can always find a transverse section accumulating at $a$ and meeting at least once (respectively twice) any regular orbits of $\hat{X}_{\lambda_0}$ in $U_a$ so Lemma~\ref{transversal} guarantees that, after shrinking $U_a$ if necessary, $\hat{\Gamma} \cap U_a=\{a\}$. Otherwise, we can consider characteristics orbits $c_0, \ldots, c_{n-1}$, with $n \geq 2$, defining a sectorial decomposition around $a$ (we follow the notation of~\cite[pp. 17--19]{dla06}). Using once again Lemma~\ref{transversal}, we note that: at each parabolic sector there may exist only one regular orbit contained in $\hat{\Gamma}$; at each hyperbolic sector, apart from shrinking $U_a$, the intersection with $\hat{\Gamma}$ can only be the characteristic orbits $c_j$ defining this sector; at each elliptic sector, apart from shrinking $U_a$ and adding two new parabolic sectors, we may suppose that the intersection of the elliptic sector with $\hat{\Gamma}$ is empty. It follows from these observations that, also in this case, $\hat{\Gamma} \cap U_a$ is a star of vertex $a$. 

Finally, if $a$ is a non-isolated singularity of $\hat{X}_{\lambda_0}$, it is well known that there exist a neighbourhood of $a$, $U_a$, an analytic map $f:U_a \to \R$ and an analytic vector field $Y$ on $U_a$ such that the restriction of $\hat{X}_{\lambda_0}$ to $U_a$ coincides with the product $f \, Y$ and the vector field $Y$ has no zeros in $U_a \setminus \{a\}$ (e. g. see \cite[Theorem~4.5]{jl07}). 

Let us denote by $Z$ the analytic set $f^{-1}(0)$ and note that, after shrinking $U_a$ if necessary, $Z$ is a star (with $a$ as vertex) decomposing $U_a$ into finitely many connected components any of which contains no singular points for $\hat{X}_{\lambda_0}$. Furthermore, by analyticity, there is no loss of generality in assuming that the neighbourhood $U_a$ has been chosen such that each branch of $Z$ is either invariant by $Y$ or a transverse section of $Y$ ({see for example \cite[Lemma~3.2]{jp09}}). {Accordingly, $\hat{\Gamma} \cap U_a$ is the union of $\{a\}$ with some of the branches of $Z$ and some regular orbits of $Y$.}

The above observation allow us to adapt the argument given in the first two cases, mutatis mutandis, to the case when $a$ is a regular point of $Y$, or when $Y$ admits a sectorial decomposition at $a$ (where we are again considering at least two characteristic orbits and among them appear at least all the branches of $Z\setminus \{a\}$ which are invariant by $Y$). We remark that the latter case includes the scenario of $a$ being a node point for $Y$.

Finally, if $a$ is a center or a focus point of $Y$, it is elementary to show that in any of the connected components of $U_a \setminus Z$ there {exists} a transverse section accumulating at $a$ and at the frontier of $U_a$. Consequently, in these two cases, it may be conclude that, after shrinking $U_a$ if necessary, $\hat{\Gamma} \cap U_a \subset Z$ and $a$ is a star point.  
\end{proof}

\begin{proof}[Proof of direct implication of Theorem \ref{mainth}]
{If $X_{\lambda_0} \not\equiv 0$, the result easily follows from Lemmas \ref{top_eq_semi} and \ref{final_lemma}. So, assume that $X_{\lambda_0} \equiv 0$.}

{The following argument is due to Roussarie \cite[Section 3]{r89} and follows an original idea of Bautin. Let $\Gamma$ be the limit periodic set and
\[
X_{\lambda}(x) = \sum_{\alpha \in \mathbb{N}^2} f_{1,\alpha}(\lambda) \,x^{\alpha}\, \partial_{x_1} +f_{2,\alpha}(\lambda)\, x^{\alpha} \,\partial_{x_2} 
\]
We consider the ideal sheaf $\mathcal{J}$ generated by $(f_{1,\alpha}(\lambda),f_{2,\alpha}(\lambda))_{\alpha \in \mathbb{N}^2}$. Note that $\lambda_0$ is in the support of this ideal by hypothesis.}

{Let $(\lambda_n)$ be the sequence of parameters converging to $\lambda_0$ and note that $\lambda_n$ is not contained in the support of $\mathcal{J}$ (because all fibres which belongs to the support of $\mathcal{J}$ are zero). Consider the monomialization $\sigma : \widetilde{\Lambda} \to \Lambda$ of the ideal sheaf $\mathcal{J}$ (see, e.g. \cite{bm08}) and denote by $(\widetilde{\lambda}_{n})$ the pre-image of $(\lambda_n)$ (which is well-defined because $(\lambda_n)$ is not in the support of $\mathcal{J}$). Since $\sigma$ is a proper map, there exists a subsequence $(\widetilde{\lambda}_{n_k})$ which converges to a parameter $\widetilde{\lambda}_0 \in \sigma^{-1}(\lambda_0)$. By construction, moreover, in a neighbourhood $U$ of $\widetilde{\lambda}_0$ there exists a multi-index $\beta \in \mathbb{N}^n$ such that $\sigma^{\ast}(X_{\lambda})|_{\mathbb{R}^2 \times U} = \widetilde{\lambda}^{\beta} \widetilde{X}_{\widetilde{\lambda}}$, where $\widetilde{X}_{\widetilde{\lambda}_0} \not\equiv 0$. We note that $\Gamma$ is the limit periodic set of $\sigma^{\ast}(X_{\lambda})$ for the parameter $\widetilde{\lambda}_0$, and that the division of the locally defined family by the monomial $\widetilde{\lambda}^{\beta}$ won't change the topology of $\Gamma$. The result follows from the first part of the proof.}

\end{proof}

\section{Construction of limit periodic sets}\label{converseproof}

\subsection{Properties of semialgebraic sets}\label{properties_semi}

We are interested in planar semialgebraic sets $\Gamma \subset \R^2$ {of dimension $0$ or $1$}. Associated to any of these sets, we introduce a free-square polynomial whose set of zeros shall play an important role in the rest of the paper. 

Let us start fixing a coordinate system for the plane $x=(x_1,x_2)$ {and let $\Gamma \subset \R^2$ be a semialgebraic set {of dimension $0$ or $1$} and whose compactification $\hat{\Gamma}$ is connected}. If $\Gamma$ is itself an algebraic set we simply take a free-squared polynomial $f_\Gamma$ making $(f_\Gamma(x)=0)=\Gamma$. Assume now that $\Gamma$ is not an algebraic set; in particular, and because $\hat{\Gamma}$ is connected, we note that none of the components of $\Gamma$ can be singletons. Let $f_{i}$ and $g_{i,j}$, $1 \leq i \leq p$ and $1 \leq j \leq q$, be polynomials such that 
\begin{equation}\label{defadhoc} \Gamma= \bigcup_{i=1}^p \bigcap_{j=1}^q \{ x \in \R^2 : \, f_{i}(x)=0 \text{ and } g_{ij}(x)>0\}. \end{equation} 

Without lost of generality, we can assume that all the $f_i$ are irreducible and also, because $\Gamma$ has empty interior, that all of them are non-constant and $(f_i(x)=0) \cap \Gamma$ is one dimensional. Using the well-known fact that any two co-prime polynomials on $\R^2$ can meet only finitely many times, it is not difficult to reason that under such conditions the polynomials $f_i$ in~\eqref{defadhoc} are uniquely defined (up to the multiplication of non-zero constants). Let us take $f_\Gamma$ as the free-square polynomial associated to the product $\prod_{i=1}^p f_{i}$; this polynomial verifies $\Gamma \subset (f_\Gamma(x)=0)$ and is uniquely defined from~\eqref{defadhoc} in the terms just expressed. In any of the two cases discussed above, we shall refer to the polynomial $f_{\Gamma}$ as \textit{the polynomial associated to $\Gamma$}. The set of zeros of $f_\Gamma$, which we shall denote by $A_\Gamma=(f_\Gamma(x)=0)$, will be also said to be \textit{the algebraic set associated to $\Gamma$}.

\begin{definition}\label{def:specialPoints} Let $\Gamma \subset \mathbb{R}^2$ be a semialgebraic set {of dimension $0$ or $1$} and such that $\hat{\Gamma}$ is connected and let $f_{\Gamma}$ and $A_\Gamma$ be its associated polynomial and algebraic set respectively. A point $a \in \Gamma$ is said to be:

(1) an \emph{algebraic} point of $\Gamma$ if there exists a neighbourhood of $a$ in $\RR$, $U$, such that $U\cap \Gamma = U \cap A_\Gamma$. We denote the set of algebraic points of $\Gamma$ by $Alg(\Gamma)$;

(2) a \textit{generic non-algebraic} point of $\Gamma$ if $a \notin Alg(\Gamma)$ and $A_{\Gamma}$ is regular at $a$ (i. e. the gradient of $f_{\Gamma}$ at $a$ is non-zero). We denote the set of generic non-algebraic points of $\Gamma$ by $Gen(\Gamma)$;

(3) a \textit{non-generic non-algebraic} point of $\Gamma$ if $a \notin Alg(\Gamma)$ and $A_{\Gamma}$ is singular at $a$ (i. e. the gradient of $f_{\Gamma}$ vanishes at $a$). We denote the set of non-generic non-algebraic point of $\Gamma$ by $NG(\Gamma)$. \end{definition}

\begin{remark}\label{rk:PointsZ1} The sets of non-algebraic points $Gen(\Gamma)$ and $NG(\Gamma)$ are both finite.\end{remark}

\begin{remark}\label{rk:PointsZ2} Let us assume that $NG(\Gamma)$ is non-empty, say $NG(\Gamma) = \{a_1, \ldots, a_r\}$ for some positive integer $r$. For every $k \in \{1,\ldots,r\}$, take a sufficiently small euclidean ball $B_k=B(a_k,\rho_k)$ centered at $a_k$ with radius $\rho_k>0$ and denote by $n_k$ the number of connected components of $\left(A_\Gamma \setminus \Gamma\right) \cap B_k$ (the number $n_k$ is the same for every sufficiently small $\rho_k>0$). By \textit{Newton-Puisseux Theorem}, for every $k \in \{1, \ldots, r\}$, there exist sequences of points $(a_i^{j,k})_{i \in \mathbb{N}} \subset A_\Gamma \setminus \Gamma$, $j \in\{1, \ldots, n_k\}$, such that each sequence is contained in a different connected component of $\left(A_\Gamma \setminus \Gamma \right) \cap B_k$ and $a_i^{j,k} \to a_j$ when $i\to \infty$. 

The following objects are used in Section \ref{nongenericcase}: the number $n_{\Gamma}=\sum_{k=1}^r n_k$; the sequence of points in $\mathbb{R}^{2 n_{\Gamma}}$, $(\alpha_i)_{i\in \mathbb{N}}$, given by $\alpha_i = (a_i^{1,1},\ldots,a_i^{n_1,1}, \allowbreak \ldots \allowbreak{, } a_i^{1,r}, \ldots, a_i^{n_r,r})$; and the limit of $(\alpha_i)_i$, $\alpha_0 \in\mathbb{R}^{2 n_{\Gamma}}$. 
\end{remark}

\begin{remark}\label{rk:PointsZ3} It follows from Remark \ref{rk:PointsZ2} that there exists a sequence of semialgebraic sets {of dimension $0$ or $1$} $(\Gamma_i)_{i\in \mathbb{N}}$ such that $\Gamma \subset \Gamma_i \subset A_{\Gamma}$, $NG(\Gamma_i)=\emptyset$ and $\Gamma_i \to \Gamma$ (in the Hausdorff topology) when $i \to \infty$. Moreover, the polynomials $f_{\Gamma}$ and the algebraic set $A_{\Gamma}$ are also the polynomial and the algebraic set associated to any of those $\Gamma_i$ and $Gen(\Gamma_i)=Gen(\Gamma) \cup \{a_i^{1,1},\ldots,a_i^{n_1,1}, \allowbreak \ldots \allowbreak{, }  a_i^{1,r}, \ldots, a_i^{n_r,r}\}$.
\end{remark}

Now assume that $\Gamma$ is compact and connected. There exists a finite number of (non-unique) points $b_1,\ldots, b_{k} \in Alg(\Gamma)$ which are regular points of the algebraic set $A_\Gamma$ and such that both $\Gamma \setminus \{b_1, \ldots, b_k\}$ and $\mathbb{R}^2 \setminus (\Gamma\setminus \{b_1, \ldots, b_k\})$ are connected. We can always fix a certain number of these points, which we call \textit{transition points}, and denote their set by $Tr(\Gamma)$. We remark that the minimal number $k$ of transition points corresponds to the number of connected components of $\mathbb{R}^2\setminus \Gamma$ minus one. Moreover, with the notation of Remark~\ref{rk:PointsZ3}, the set $Tr(\Gamma)$ is a valid set of transition points for $\Gamma_i$, for all $i$ sufficiently big. 

\subsection{Construction of generic compact limit periodic sets}\label{genericcase}

Let us fix a connected and compact semialgebraic set $\Gamma \subset \mathbb{R}^2$ {of dimension $0$ or $1$} and such that $NG(\Gamma) = \emptyset$. Let $f=f_{\Gamma}$ be the polynomial associated to $\Gamma$ and $A_{\Gamma}$ its set of zeros and fix a transition set for $\Gamma$, $Tr(\Gamma)$. Fix a coordinate system $x= (x_1,x_2)$ of $\mathbb{R}^2$ and a parameter $\lambda \in\mathbb{R}$. Denote by $S$ the set $Gen(\Gamma) \cup Tr(\Gamma)$, which is a finite set. 

We consider the function
\[
\begin{aligned}
h(x,\lambda)&= f(x)^2 - \lambda \prod_{p \in S} \left( \|x-p\|^2 - \lambda^2\right),
\end{aligned}
\] where $\| \cdot \|$ stands for the euclidean norm on $\RR$, and the polynomial family of planar vector fields $(X_{\lambda})_{\lambda \in \R}$ where
\begin{equation}\label{genVF} X_{\lambda} =\left( \frac{\partial h}{\partial x_2} + h \frac{\partial h}{\partial x_1}\right)  \partial_{x_1} + \left(-\frac{\partial h}{\partial x_1}  + h \frac{\partial h}{\partial x_2} \right) \partial_{x_2}. \end{equation}

We devote the rest of the Section to show $\Gamma$ is a limit periodic set of $(X_\lambda)_{\lambda \in \R}$ at $\lambda=0$. The key to achieve it is to understand how the level curves (in respect to the parameter $\lambda$) of $h$ are. We shall start giving a local description of $(h(x,\lambda)=0)$ in a neighbourhood of a point $(a,0)$ where $a \in \Gamma$; we treat separately the cases $a \in Alg(\Gamma) \setminus Tr(\Gamma)$ {(Lemma \ref{lem:1}) and $a \in Gen(\Gamma) \cup Tr(\Gamma)$ (Lemma \ref{lem:2}).

Here and subsequently, given any set $A \subset \R^2 \times \R$ and any $t \in \R$ we will denote $A \cap (\lambda=t) =\{(x,\lambda) \in A : \lambda=t\}$; when convenient, we will also understand that $A \cap (\lambda=t)$ is identified with $\{x \in \R^2 : (x,t) \in A\}$. In particular, $Z_{t}$ which stands for the level curve $(h(x,\lambda)=0) \cap (\lambda=t)$ will repeatedly be seen as a subset of $\RR$.

\begin{lemma}\label{lem:1}
For every $a \in Alg(\Gamma) \setminus Tr(\Gamma)$ there exist a number $\epsilon_a>0$ and a compact neighbourhood $V_a$ of $a$ such that $Z_t \cap V_a \subset \RR \setminus \Gamma$ for every $0<t<\epsilon_a$. Moreover, for any connected component $W$ of $V_{a} \setminus \Gamma$ and any $0<t<\epsilon_a$, $Z_t \cap W$ is a non-empty connected regular curve which converges (in the Hausdorff topology) to $\Cl(W) \cap \Gamma$, when $t$ tends to $0$ (see~Figure~\ref{regsing}).
\end{lemma}
\begin{proof}
Let us start considering a number $\epsilon_a>0$, a compact neighbourhood $V_a \subset \R^2$ of $a$ and the coordinate system $z= x-a$ (which is centered at $a$)  such that $h(z,\lambda) = f(z)^2 - \lambda u(z,\lambda)$, where $u(z,\lambda) > 0$ at all points in $V_a \times (\epsilon_a,-\epsilon_a)$. 

By the \textit{implicit function theorem}, we may assume that there exists an analytic function $\lambda: V_a \to \R$ such that $h(z,\lambda(z))=0$ for every $z \in V_a$.

We note that the curves $Z_t^a=Z_t \cap V_a$ correspond to the $t$-level curves of $\lambda(z)$, that is, $Z_t^a = (\lambda(z) = t)$. By continuity of $\lambda(z)$, shrinking $V_a$ if necessary, this implies that $Z_t^a$ converges (in the Hausdorff topology) to $\Gamma \cap V_a=(\lambda(z) = 0)$. Furthermore, since the level curves of non-constant analytic functions (restricted to a compact set) are generically regular, we conclude that $Z_t^a$ are regular for all $t>0$ sufficiently small.

Next, since $\lambda(z)\geq 0$ for every $z \in V_a$, we conclude that, for every component $W$ of $V_{a} \setminus \Gamma$, $Z_t^a \cap W$ is non-empty for all small enough $t>0$.

Finally, fix a component $W$ of $V_{a} \setminus \Gamma$ and suppose by contradiction that there exists a sequence $(t_n)_n$ converging to $0$ and such that $Z_{t_n}^a\cap W$ is not connected.  Without loss of generality, we may suppose that $\Cl(W)$ is a compact semialgebraic set. Denote by $\Gamma_W$ the semialgebraic set $\Cl(W) \cap \Gamma$. By the curve selection Lemma (see for example \cite[Lemma~3.1]{M68}), there exists an analytic curve $\phi: [0,1] \to \Cl(W)$ such that $\phi(1)=a \in \Gamma_W$, $\phi(t) \in \Cl(W) \setminus \Gamma_W$ for all $t\neq 0$ and $\Cl(W) \setminus \phi([0,1])$ is not connected. Since all connected components of $Z_t^a$ must converge to $\Gamma_W$, we conclude that the curve $\phi([0,1])$ intersects each of the components of $Z_t^a$. This implies that the function $\lambda \circ \phi$ is constant and equal to $0$ (the value at $\phi(1)$), which is a contradiction.\end{proof}
\begin{figure}
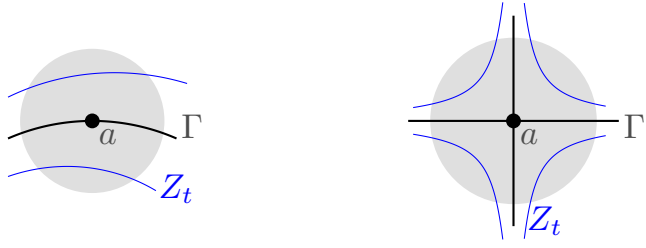


\caption{A regular point (left) and a singular point (right).}
\label{regsing}
\end{figure}

\begin{lemma}\label{lem:2}
For every $a \in Gen(\Gamma) \cup Tr(\Gamma)$, there exist a neighbourhood $V_{a}$ of $a$, a positive $\epsilon_a >0$ and a coordinate system $(y,\lambda)$ defined on $V_{a} \times (-\epsilon_a,\epsilon_a)$ and centered at $(a,0)$ such that $A_\Gamma \cap V_a =(y_1 = 0)$ and
\begin{equation}\label{normalGeneric} h(y,\lambda) =  u(y,\lambda) \left[ y_1^2 - \lambda \left( y_2^2 - \lambda^2 \right) \right] \end{equation} where $u(y,\lambda)$ is a unit over $V_{a} \times (-\epsilon_a,\epsilon_a)$ (see~Figure~\ref{transgen}). 
\end{lemma}
\begin{proof}
Consider the coordinate system $z= x-a$ (which is centered at $a$) and note that in a sufficiently small neighbourhood of $(a,0)$ of the form $U_{a}=V_{a} \times (-\epsilon_a,\epsilon_a)$, we can write
\[
h(z,\lambda) = f(z)^2 - \lambda \left[z_1^2+z_2^2 - \lambda^2 \right] u(z,\lambda)
\]
where $u(z,\lambda) > 0$ at all points in $U_{a}$. Apart from shrinking $U_{a}$, we can suppose that $\nabla f(z) \neq 0$ at all points in $U_{a}$. Therefore (apart from a preliminary rotation) the change of coordinates $\widetilde{y}_1 = u(z,\lambda)^{-\frac{1}{2}} f(z) = \xi z_1 + \psi(z,\lambda)$ (where $\xi\neq 0$ and $\psi(z,\lambda)$ has order at least two) and $\widetilde{y}_2=z_2$ is an isomorphism on $U_{a}$. We get
\[
h(\widetilde{y},\lambda) = u(\widetilde{y},\lambda) \left( \widetilde{y}_1^2 - \lambda \left(\widetilde{y}_1^2  + \widetilde{y}_2^2v \left(\widetilde{y},\lambda\right)  - \lambda^2 \right) \right)
\]
where $v(\widetilde{y},\lambda)$ is an analytic function such that $v(0,0)> 0$. Finally, apart from shrinking $U_{a}$, the change of coordinates $y_1= \widetilde{y}_1 \sqrt{1+\lambda}$ and $y_2 =\widetilde{y}_2 \sqrt{v(\widetilde{y},\lambda)}$ is an isomorphism making
\[
h(y,\lambda) =  u(y,\lambda) \left[ y_1^2 - \lambda \left( y_2^2 - \lambda^2 \right) \right]
\] and $(y_1 = 0)=(\widetilde{y}_1=0)=(f(z)=0) \cap V_a$ as we wanted to prove.
\end{proof}
\begin{figure}
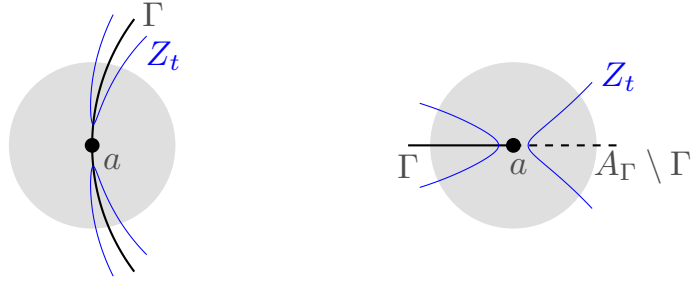


\caption{ A transition point (left) and a generic point (right).}
\label{transgen}
\end{figure}

\begin{proposition}\label{prop:CompactNonGeneric} There exist an open neighbourhood $U$ of $\Gamma \times \{0\}$ and a number $\epsilon>0$ such that, for every $0<t<\epsilon$, $Z_t\cap U$ contains a compact and regular connected component $\gamma_t$ such that $\gamma_t \to \Gamma$ (in the Hausdorff topology) when $t\to 0$. 
\end{proposition}
\begin{proof}
For every $a \in \Gamma$ take a neighbourhood $V_a$ of $a$ and a number $\epsilon_a>0$ as in Lemma~\ref{lem:1} or \ref{lem:2}. The compacity of $\Gamma$ allows us to take a relatively compact open neighbourhood $U$ of $\Gamma \times \{0\}$, of the form $U = V \times (-\delta,\delta)$ with $V \subset \mathbb{R}^2$ and $\delta>0$, such that $U \subset \bigcup_{a \in \Gamma} U_{a}$. 

Note that, from the two previous lemmas, we can assume that $Z_t \cap U$ is regular for every sufficiently small $t >0$. Also, the continuity of $h$ guarantees that $Z_t \cap U$ converges to $A_{\Gamma} \cap V$ when $t$ tends to $0$ (in the Hausdorff topology). 

Let us fix a point $b \in Alg(\Gamma) \setminus Tr(\Gamma)$ and $W$ a component of $V_b \setminus \Gamma$. For every sufficiently small $t>0$, let us call $\gamma_t$ the connected component of $Z_t$ which meets $W$ (see Lemma \ref{lem:1}). Let $\gamma_0 \subset A_{\Gamma}$ denote the limit of $\gamma_t$ when $t\to 0$ (which contains $b\in \Gamma$). We are then left with the task of proving that $\gamma_0=\Gamma$. 

We start showing that $\gamma_0 \subset \Gamma$. We proceed by contradiction assuming the existence of a point $c \in \gamma_0 \setminus \Gamma$. After shrinking $V$ and $V_a$ for $a\in Gen(\Gamma)$ if necessary, we may suppose that the points $b$ and $c$ lies in different connected components of the set $R=V \setminus \bigcup_{a \in Gen(\Gamma)} V_a$. In particular, $\gamma_t  \cap R$ is disconnected and these disconnected components can only join each other by passing through one of the open sets $U_a$ with $a\in Gen(\Gamma)$. This leads to contradiction with Lemma~\ref{lem:2}.

Since $\gamma_t$ is a connected regular curve, $\RR \setminus \gamma_t$ must consists in exactly two connected components, say $C^1_t$ and $C^2_t$. Now, for every $\epsilon_0>0$, consider the set $\gamma_{\epsilon_0} = \gamma_0 \setminus \cup_{a\in Tr(\Gamma)} B(a,\epsilon_0)$.  We claim that, for every small enough $t>0$, \begin{equation}\label{cl:2} \gamma_{\epsilon_0} \subset \mathbb{R}^2 \setminus \gamma_t.\end{equation} Indeed, let us suppose that $\gamma_{\epsilon_0}$ meets both $C^1_t$ and $C^2_t$ for all small $t>0$. Since $\gamma_t$ can only cross points of $\Gamma$ near $a\in Tr(\Gamma) \cup Gen(\Gamma)$, we conclude that there exists a point $a\in Tr(\Gamma)$ such that $\gamma_t$ crosses $\Gamma \cap V_a$ in such a way that $\gamma_{\epsilon_0} \cap C^1_t \cap V_a$ and $\gamma_{\epsilon_0} \cap C^2_t \cap V_a$ $\Gamma$ are both non-empty. But this gives us again a contradiction with Lemma \ref{lem:2}.

Finally, let us denote by $W_i$, $i=1,\ldots k$, the connected components of $\mathbb{R}^2 \setminus \Gamma$ and set $I$ as the subset of indexes $i \in \{1,\ldots,k\}$ such that $W_i \cap \gamma_t \neq \emptyset$ for all sufficiently small $t>0$ (note that, by construction, $I$ is non-empty). The proof is completed by showing that $I=\{1,\ldots,k\}$ and $\gamma_0=\cup_{i\in I} {\left(\Cl(W_i)\setminus W_i\right)}$. We argue in two steps. 

First, suppose by contradiction that $\gamma_0 \neq \cup_{i\in I} {\left(\Cl(W_i)\setminus W_i\right)}$. Without restriction of generality, we can suppose that $1\in I$ and $ \gamma_0 \cap (\Cl({W_1})\setminus W_1) \neq \Cl(W_1)\setminus W_1$. We consider a point $c \in \Cl(W_1)\setminus W_1$ which does not belong to $\gamma_0$ and an analogous family of ovals $\alpha_t \subset Z_t$ constructed in the same way as $\gamma_t$ but in respect to $c$ and the connected set $W_1$. By \eqref{cl:2}, we conclude that $\alpha_0 \cap \gamma_0$ can only contain points which lie in $Tr(\Gamma)$. This implies that $\Gamma \setminus Tr(\Gamma)$ has at least two disconnected components $\gamma_0 \setminus Tr(\Gamma)$ and $\alpha_0 \setminus Tr(\Gamma)$, which is in contradiction with the definition of $Tr(\Gamma)$.

Next, suppose, by contradiction, that $\gamma_0 = \cup_{i\in I}{\left(\Cl(W_i)\setminus W_i\right)}$ but $I \neq \{1,\ldots, k\}$. Denote by $\Gamma_0 =  \cup_{i\notin I} {\left(\Cl(W_i)\setminus W_i\right)}$. Since the level curve $Z_t$ can only cross points of $\Gamma$ near $a \in Tr(\Gamma) \cup Gen(\Gamma)$ (see Lemmas \ref{lem:1} and \ref{lem:2}), we conclude that $\gamma_0 \cap \Gamma_0 \cap Tr(\Gamma) = \emptyset$. Therefore, $\gamma_0 \cap \Gamma_0 \subset \Gamma \setminus Tr(\Gamma)$ disconnects $\mathbb{R}^2$, which is again in contradiction with the choice of $Tr(\Gamma)$.

We conclude the proof by remarking that, since $\gamma_t$ converges to $\Gamma$, for small enough $t>0$, $\gamma_t$ must be a compact set contained in the interior of $U$. \end{proof}

After noticing that any compact and regular connected component of a planar algebraic set is a topological circle \cite[Theorem~2, p. 180]{KU}, it follows from Proposition~\ref{prop:CompactNonGeneric} that the polynomial family of planar vector fields given by~\eqref{genVF} has $\Gamma$ as a limit periodic set at $\lambda=0$. Indeed, it is enough to prove that, for every sufficiently small $t>0$, the topological circle $\gamma_t \subset \RR$ given by Proposition~\ref{prop:CompactNonGeneric} is a {limit cycle} of $X_t$. To show this, it suffices to note that $${X_{\lambda}(h)=}\left(\frac{\partial h}{\partial x_2} + h \frac{\partial h}{\partial x_1}\right) \frac{\partial h}{\partial_{x_1}} + \left(-\frac{\partial h}{\partial x_1}  + h \frac{\partial h}{\partial x_2}\right) \frac{\partial h}{\partial_{x_2}} = h \left\| \nabla h \right\|^2$$ and, as a consequence, $Z_t$ is an invariant set containing any periodic orbit of $X_t$. Finally, the fact that $\gamma_t$ is regular guarantees that it is a periodic orbit. This proves the converse of Theorem \ref{mainth} (under the extra assumption that $\Gamma$ is compact and generic).

\subsection{Construction of non-generic compact limit periodic sets}\label{nongenericcase}

Let us now fix a connected and compact semialgebraic set $\Gamma \subset \mathbb{R}^2$ {of dimension $0$ or $1$} and $NG(\Gamma) \neq \emptyset$. Denote by $f=f_{\Gamma}$ the polynomial associated to $\Gamma$, $A_{\Gamma}$ the associated algebraic set and $S = Gen(\Gamma) \cup Tr(\Gamma)$. Fix a coordinate system $x= (x_1,x_2)$ of $\mathbb{R}^2$ and parameters $(\alpha,\lambda) \in \mathbb{R}^{2n_{\Gamma}+1}$ where $\alpha = (\alpha^1, \ldots, \alpha^n) \in \mathbb{R}^{2n_{\Gamma}}$. We consider the function
\[
\begin{aligned}
h(x,\alpha,\lambda)&= f(x)^2 - \lambda \prod_{p \in S} \left( \|x-p\|^2 - \lambda^2\right) \prod_{i=1}^n \left( \|x-\alpha^i\|^2 - \lambda^2\right).
\end{aligned}
\]

Let us take the number $n_{\Gamma}$, the sequence $(\alpha_i)_{i \in \N}$ and the point $\alpha_0$ as in Remark~\ref{rk:PointsZ2} and, for every $i \in \mathbb{N}$, let us consider $h_i(x,\lambda) = h(x,\alpha_i,\lambda)$. For any $i \in \N$, we can apply Proposition~\ref{prop:CompactNonGeneric} to the semialgebraic set $\Gamma_i$ introduced in Remark~\ref{rk:PointsZ3} to deduce that there exists a value $0<\lambda_i < \frac{1}{i}$ such that the level set $(h_i(x,\lambda) =0) \cap (\lambda=\lambda_i)$ contains a subset $\gamma_i$ which is regular connected, compact and $\frac{1}{i}$-close (in respect to the Hausdorff topology) to $\Gamma_i$. Furthermore, apart from shrinking $\lambda_i$ if necessary, we can suppose that $\gamma_i \cap NG(\Gamma) = \emptyset$, because $NG(\Gamma) \subset Alg(\Gamma_i)$ and Lemma \ref{lem:1}. In particular, note that $\gamma_i \to \Gamma$ when $i\to \infty$ since $\Gamma_i$ converges to $\Gamma$.

\begin{remark}\label{rk:limitcycles} We note that, for every point $a\in \Gamma$, there exists $N >0$ such that $\gamma_i \cap \{a\} = \emptyset$ for every $i>N$. Indeed, by construction $\gamma_i$ only crosses $\Gamma_i \supset \Gamma$ near the points $Tr(\Gamma) \cup Gen(\Gamma_i)$. So, assuming by contradiction that there exists a point $a \in \Gamma$ so that $\gamma_i \cap \{a\} \neq \emptyset$ for an infinite number of $i$, we conclude that $a \in Tr(\Gamma) \cup Gen(\Gamma) \cup NG(\Gamma)$. Next, by Lemma \ref{lem:2} we conclude that $a\in NG(\Gamma)$, which contradicts the choice of $\lambda_i$.
\end{remark}

It follows from the above considerations (just as in the previous Section) that the algebraic family of vector fields
$$ X_{\alpha,\lambda} = \left( \frac{\partial h}{\partial {x_2}} + h \frac{ \partial h}{\partial x_1}\right)  \partial_{x_1} + \left(- \frac{ \partial h}{\partial {x_1}} h  + h \frac{ \partial h}{\partial {x_2}} \right) \partial_{x_2} $$
has $\Gamma$ as a limit periodic set at $(\alpha,\lambda)=(\alpha_0,0)$.

\subsection{Construction of unbounded limit periodic sets}\label{unboundedcase}

Finally, let $\Gamma \subset \mathbb{R}^2$ be a closed and unbounded semialgebraic set {of dimension $0$ or $1$} whose compactification $\hat{\Gamma}$ is connected. Apart from considering a translation of $\mathbb{R}^2$, we can assume that $(0,0) \notin \Gamma$. 

Let us consider the transition map of the Bendixson compactification $\phi: \mathbb{R}^2\setminus \{0\} \to \mathbb{R}^2$ given by $\phi(x_1,x_2) = (x_1/r,-x_2/r) $ where $r= x_1^2+x_2^2$ (see Section~\ref{bendixson}).

Note that $\phi(\Gamma)$ is a semialgebraic set (by e. g.~\cite[Corollary~1.8]{bm88}), whose closure $Z = \phi(\Gamma) \cup \{0\}$ is a compact and connected semialgebraic set {of dimension $0$ or $1$}. By the previous Sections, there exist a polynomial family of planar vector fields $(Y_{\lambda})_{\lambda \in \Lambda}$ and a parameter $\lambda_0$ such that $Z$ is a limit periodic set for the family $(Y_{\lambda})_{\lambda}$ at $\lambda_0$. We denote by $(z_{\lambda_n})_n$ the sequence of limit cycles of $Y_{\lambda}$ which converge to $Z$. 

Let us now consider the map $\Phi: (\mathbb{R}^2\setminus \{0\}) \times \Lambda \to \mathbb{R}^2 \times \Lambda$ given by $\Phi(x_1,x_2,\lambda) = (\phi(x_1,x_2),\lambda)$. The pull-back $\Phi^{\ast}(Y_{\lambda})$ is rational and there exists an integer $d \geq 0$ such that $X_{\lambda} = (x_1^2+x_2^2)^d \Phi^{\ast}(Y_{\lambda})$ is a polynomial family of vector fields. According to Remark~\ref{rk:limitcycles}, for every sufficiently big $n$, $z_{\lambda_n}$ does not intersect the origin so $\Phi^{-1}(z_{\lambda_n})$ is itself a limit cycles of $X_{\lambda_n}$. It follows from the construction that $\Gamma$ is a limit periodic set of $(X_{\lambda})_{\lambda}$ at $\lambda_0$.

\bibliographystyle{alpha}

\end{document}